\documentclass[reqno,a4paper]{amsart}
\usepackage{amssymb}
\usepackage{amsmath}
\usepackage{amsfonts}

\newtheorem{theorem}{Theorem}[section]
\theoremstyle{plain}

\newtheorem{definition}{Definition}[section]

\numberwithin{equation}{section}
\input{tcilatex}

\begin{document}
\title[]{ON STRONG $r$-HELIX SUBMANIFOLDS AND SPECIAL CURVES}
\author{Evren Z\i plar}
\address{Department of Mathematics, Faculty of Science, University of
Ankara, Tando\u{g}an, Turkey}
\email{evrenziplar@yahoo.com}
\urladdr{}
\author{Ali \c{S}enol}
\address{Department of Mathematics, Faculty of Science, \c{C}ank\i r\i\
Karatekin University, \c{C}ank\i r\i , Turkey}
\email{asenol@karatekin.edu.tr}
\author{Yusuf Yayl\i }
\address{Department of Mathematics, Faculty of Science, University of
Ankara, Tando\u{g}an, Turkey}
\email{yayli@science.ankara.edu.tr}
\thanks{}
\urladdr{}
\date{}
\keywords{Strong $r$-helix submanifold; Line of curvature; Geodesic curve;
Slant helix.\\
Corresponding author: Evren Z\i plar, e-mail: evrenziplar@yahoo.com}
\thanks{}

\begin{abstract}
In this paper, we investigate special curves on a strong $r$-helix
submanifold in Euclidean $n$-space $E^{n}$. Also, we give the important
relations between strong $r$-helix submanifolds and the special curves such
as line of curvature, geodesic and slant helix.
\end{abstract}

\maketitle

\section{Introduction}

In differential geometry of manifolds, an helix submanifold of $IR^{n}$ with
respect to a fixed direction $d$ in $IR^{n}$ is defined by the property that
tangent planes make a constant angle with the fixed direction $d$ (helix
direction) in [5]. Di Scala and Ruiz-Hern\'{a}ndez have introduced the
concept of these manifolds in [5]. Besides, the concept of strong $r$-helix
submanifold of $IR^{n}$ was introduced in [4]. Let $M\subset IR^{n}$ be a
submanifold and let $H(M)$ be the set of helix directions of $M$. We say
that $M$ is a strong $r$-helix if the set $H(M)$ is a linear subspace of $%
IR^{n}$ of dimension greater or equal to $r$ in [4].

Recently, M. Ghomi worked out the shadow problem given by H.Wente. And, He
mentioned the shadow boundary in [8]. Ruiz-Hern\'{a}ndez investigated that
shadow boundaries are related to helix submanifolds in [12].

Helix hypersurfaces has been worked in nonflat ambient spaces in [6,7].
Cermelli and Di Scala have also studied helix hypersurfaces in liquid
cristals in [3].

The plan of this paper is as follows. In section 2, we mention some basic
facts in the general theory of strong $r$-helix, manifolds and curves. And,
in section 3, we give the important relations between strong $r$-helix
submanifolds and some special curves such as line of curvature, geodesic and
slant helix.

\section{PRELIMINARIES}

\begin{definition}
Let $M\subset IR^{n}$ be a submanifold of a euclidean space. A unit vector $%
d\in IR^{n}$ is called a helix direction of $M$ if the angle between $d$ and
any tangent space $T_{p}M$ is constant. Let $H(M)$ be the set of helix
directions of $M$. We say that $M$ is a strong $r$-helix if $H(M)$ is a $r$%
-dimensional linear subspace of $IR^{n}$ [4].
\end{definition}

\begin{definition}
A submanifold $M\subset IR^{n}$ is a strong $r$-helix if the set $H(M)$ is a
linear subspace of $IR^{n}$ of dimension greater or equal to $r$ [4].
\end{definition}

\begin{definition}
A unit speed curve $\alpha :I\rightarrow E^{n}$ is called a slant helix if
its unit principal normal $V_{2}$ makes a constant angle with a fixed
direciton $U$ [1].
\end{definition}

\begin{definition}
Let the $(n-k)$-manifold $M$ be submanifold of the Riemannian manifold $%
\overline{M}=E^{n}$ and let $\overline{D}$ be the Riemannian connexion on $%
\overline{M}=E^{n}$. For $C^{\infty \text{ }}$fields $X$ and $Y$ with domain 
$A$ on $M$ (and tangent to $M$), define $D_{X}Y$ and $V(X,Y)$ on $A$ by
decomposing $\overline{D}_{X}Y$ into unique tangential and normal
components, respectively; thus,%
\begin{equation*}
\overline{D}_{X}Y=D_{X}Y+V(X,Y)\text{. }
\end{equation*}%
Then, $D$ is the Riemannian connexion on $M$ and $V$ is a symmetric
vector-valued 2-covariant $C^{\infty \text{ }}$tensor called the second
fundamental tensor. The above composition equation is called the Gauss
equation [9].
\end{definition}

\begin{definition}
Let the $(n-k)$-manifold $M$ be submanifold of the Riemannian manifold $%
\overline{M}=E^{n}$ , let $\overline{D}$ be the Riemannian connexion on $%
\overline{M}=E^{n}$ and let $D$ be the Riemannian connexion on $M$. Then,
the formula of Weingarten%
\begin{equation*}
\overline{D}_{X}N=-A_{N}(X)+D_{X}^{\bot }N
\end{equation*}%
for every $X$ and $Y$ tangent to $M$ and for every $N$ normal to $M$. $A_{N}$
is the shape operator associated to $N$ also known as the Weingarten
operator corresponding to $N$ and $D^{\bot }$ is the induced connexion in
the normal bundle of $M$ ($A_{N}(X)$ is also the tangent component of $-%
\overline{D}_{X}N$ and will be denoted by $A_{N}(X)$ $=$tang($-\overline{D}%
_{X}N$)). Specially, if $M$ is a hypersurface in $E^{n}$, we have $%
\left\langle V(X,Y),N\right\rangle =\left\langle A_{N}(X),Y\right\rangle $
for all $X$, $Y$ tangent to $M$. So,%
\begin{equation*}
V(X,Y)=\left\langle V(X,Y),N\right\rangle N=\left\langle
A_{N}(X),Y\right\rangle N
\end{equation*}%
and we obtain%
\begin{equation*}
\overline{D}_{X}Y=D_{X}Y+\left\langle A_{N}(X),Y\right\rangle N\text{ .}
\end{equation*}%
For this definition 2.5, note that the shape operator $A_{N}$ is defined by
the map $A_{N}:\varkappa (M)\rightarrow \varkappa (M)$, where $\varkappa (M)$
is the space of tangent vector fields on $M$ and if $p\in M$, the shape
operator $A_{N}$ is defined by the map $A_{p}:T_{p}(M)\rightarrow T_{p}(M)$%
.The eigenvalues of $A_{p}$ are called the principal curvatures (denoted by $%
\lambda _{i}$) and the eigenvectors of $A_{p}$ are called the principal
vectors [10,11].
\end{definition}

\begin{definition}
If $\alpha $ is a (unit speed) curve in $M$ with $C^{\infty \text{ }}$unit
tangent $T$, then $V(T,T)$ is called normal curvature vector field of $%
\alpha $ and $k_{T}=\left\Vert V(T,T)\right\Vert $ is called the normal
curvature of $\alpha $ [9].
\end{definition}

\section{Main Theorems}

\begin{theorem}
\textbf{\ }Let $M$ be a strong $r$-helix hypersurface and $H(M)\subset E^{n}$
be the set of helix directions of $M$. If $\alpha :I\subset IR\rightarrow M$
is a (unit speed) line of curvature (not a line) on $M$, then $d_{j}\notin
Sp\left\{ N,T\right\} $ along the curve $\alpha $ for all $d_{j}\in H(M)$,
where $T$ is the tangent vector field of $\alpha $ and $N$ is a unit normal
vector field of $M$.
\end{theorem}

\begin{proof}
We assume that $d_{j}\in Sp\left\{ N,T\right\} $ along the curve $\alpha $
for any $d_{j}\in H(M)$. Then, along the curve $\alpha $, since $M$ is a
strong $r$-helix hypersurface, we can decompose $d_{j}$ in tangent and
normal components:%
\begin{equation}
d_{j}=\cos (\theta _{j})N+\sin (\theta _{j})T
\end{equation}%
where $\theta _{j}$ is constant. From (3.1),by taking derivatives on both
sides along the curve $\alpha $, we get:%
\begin{equation}
0=\cos (\theta _{j})N^{%
{\acute{}}%
}+\sin (\theta _{j})T^{%
{\acute{}}%
}\text{ }
\end{equation}%
Moreover, since $\alpha $ is a line of curvature on $M$,%
\begin{equation}
N%
{\acute{}}%
=\lambda \alpha 
{\acute{}}%
\text{ }
\end{equation}%
along the curve $\alpha $. By using the equations (3.2) and (3.3), we deduce
that the system $\left\{ \alpha 
{\acute{}}%
,T^{%
{\acute{}}%
}\right\} $ is linear dependent. But, the system $\left\{ \alpha 
{\acute{}}%
,T^{%
{\acute{}}%
}\right\} $ is never linear dependent. This is a contradiction. This
completes the proof.
\end{proof}

\begin{theorem}
Let $M$ be a submanifold with $(n-k)$ dimension in $E^{n}$. Let $\overline{D}
$ be Riemannian connexion (standart covariant derivative) on $E^{n}$ and $D$
be Riemannian connexion on $M$. Let us assume that $M\subset E^{n}$ be a
strong $r$-helix submanifold and $H(M)\subset E^{n}$ be the space of the
helix directions of $M$. If $\alpha :I\subset IR\rightarrow M$ is a (unit
speed) geodesic curve on $M$ and if $\left\langle V_{2},\xi
_{j}\right\rangle $ is a constant function along the curve $\alpha $, then $%
\alpha $ is a slant helix in $E^{n}$, where $V_{2}$ is the unit principal
normal of $\alpha $ and $\xi _{j}$ is the normal component of a direction $%
d_{j}\in H(M)$.
\end{theorem}

\begin{proof}
Let $T$ be the unit tangent vector field of $\alpha $. Then, from the
formula Gauss in Definition (2.4),%
\begin{equation}
\overline{D}_{T}T=D_{T}T+V(T,T)
\end{equation}%
According to the Theorem, since $\alpha $ is a geodesic curve on $M$,%
\begin{equation}
D_{T}T=0
\end{equation}%
So, by using (3.4),(3.5) and Frenet formulas, we have:%
\begin{equation*}
\overline{D}_{T}T=k_{1}V_{2}=V(T,T)
\end{equation*}%
That is, the vector field $V_{2}\in \vartheta (M)$ along the curve $\alpha $%
, where $\vartheta (M)$ is the normal space of $M$. On the other hand, since 
$M$ is a strong $r$-helix submanifold, we can decompose any $d_{j}\in H(M)$
in its tangent and normal components:%
\begin{equation}
d_{j}=\cos (\theta _{j})\xi _{j}+\sin (\theta _{j})T_{j}
\end{equation}%
where $\theta _{j}$ is constant. Moreover, according to the Theorem, $%
\left\langle V_{2},\xi _{j}\right\rangle $ is a constant function along the
curve $\alpha $ for the normal component $\xi _{j}$ of \ a direction $%
d_{j}\in H(M)$. Hence, doing the scalar product with $V_{2}$ in each part of
the equation (3.6), we obtain: 
\begin{equation}
\left\langle d_{j},V_{2}\right\rangle =\cos (\theta _{j})\left\langle
V_{2},\xi _{j}\right\rangle +\sin (\theta _{j})\left\langle
V_{2},T_{j}\right\rangle
\end{equation}%
Since $\cos (\theta _{j})\left\langle V_{2},\xi _{j}\right\rangle =$constant
and $\left\langle V_{2},T_{j}\right\rangle =0$ ( $V_{2}\in \vartheta (M)$)
along the curve $\alpha $, from (3.7) we have:%
\begin{equation*}
\left\langle d_{j},V_{2}\right\rangle =\text{constant.}
\end{equation*}%
along the curve $\alpha $. Consequently, $\alpha $ is a slant helix in $%
E^{n} $.
\end{proof}

\begin{theorem}
Let $M$ be a submanifold with $(n-k)$ dimension in $E^{n}$. Let $\overline{D}
$ be Riemannian connexion (standart covariant derivative) on $E^{n}$ and $D$
be Riemannian connexion on $M$. Let us assume that $M\subset E^{n}$ be a
strong $r$-helix submanifold and $H(M)\subset E^{n}$ be the space of the
helix directions of $M$. If $\alpha :I\subset IR\rightarrow M$ is a (unit
speed) curve on $M$ with the normal curvature function $k_{T}=0$ and if $%
\left\langle V_{2},T_{j}\right\rangle $ is a constant function along the
curve $\alpha $, then $\alpha $ is a slant helix in $E^{n}$, where $V_{2}$
is the unit principal normal of $\alpha $ and $T_{j}$ is the tangent
component of a direction $d_{j}\in H(M)$.
\end{theorem}

\begin{proof}
Let $T$ be the unit tangent vector field of $\alpha $. Then, from the
formula Gauss in Definition (2.4),%
\begin{equation}
\overline{D}_{T}T=D_{T}T+V(T,T)
\end{equation}%
According to the Theorem, since the normal curvature $k_{T}=0$,%
\begin{equation}
V(T,T)=0
\end{equation}%
So, by using (3.8),(3.9) and Frenet formulas, we have:%
\begin{equation*}
\overline{D}_{T}T=k_{1}V_{2}=D_{T}T\text{.}
\end{equation*}%
That is, the vector field $V_{2}\in T_{\alpha (t)}M$, where $T_{\alpha (t)}M$
is the tangent space of $M$. On the other hand, since $M$ is a strong $r$%
-helix submanifold, we can decompose any $d_{j}\in H(M)$ in its tangent and
normal components:%
\begin{equation}
d_{j}=\cos (\theta _{j})\xi _{j}+\sin (\theta _{j})T_{j}
\end{equation}%
where $\theta _{j}$ is constant. Moreover, according to the Theorem, $%
\left\langle V_{2},T_{j}\right\rangle $ is a constant function along the
curve $\alpha $ for the tangent component $T_{j}$ of \ a direction $d_{j}\in
H(M)$. Hence, doing the scalar product with $V_{2}$ in each part of the
equation (3.10), we obtain: 
\begin{equation}
\left\langle d_{j},V_{2}\right\rangle =\cos (\theta _{j})\left\langle
V_{2},\xi _{j}\right\rangle +\sin (\theta _{j})\left\langle
V_{2},T_{j}\right\rangle
\end{equation}%
Since $\sin (\theta _{j})\left\langle V_{2},T_{j}\right\rangle =$constant
and $\left\langle V_{2},\xi _{j}\right\rangle =0$ ($V_{2}\in T_{\alpha (t)}M$%
) along the curve $\alpha $, from (3.11) we have:%
\begin{equation*}
\left\langle d_{j},V_{2}\right\rangle =\text{constant.}
\end{equation*}%
along the curve $\alpha $. Consequently, $\alpha $ is a slant helix in $%
E^{n} $.
\end{proof}

\begin{definition}
Given an Euclidean submanifold of arbitrary codimension $M\subset IR^{n}$. A
curve $\alpha $ in $M$ is called a line of curvature if its tangent $T$ is a
principal vector at each of its points. In other words, when $T$ (the
tangent of $\alpha $) is a principal vector at each of its points, for an
arbitrary normal vector field $N\in \vartheta (M)$, the shape operator $%
A_{N} $ associated to $N$ says $A_{N}(T)=$tang$(-$ $\overline{D}%
_{T}N)=\lambda _{j}T$ along the curve $\alpha $, where $\lambda _{j}$ is a
principal curvature and $\overline{D}$ be the Riemannian connexion(standart
covariant derivative) on $IR^{n}$ [2].
\end{definition}

\begin{theorem}
Let $M$ be a submanifold with $(n-k)$ dimension in $E^{n}$ and let $%
\overline{D}$ be Riemannian connexion (standart covariant derivative) on $%
E^{n}$. Let us assume that $M\subset E^{n}$ be a strong $r$-helix
submanifold and $H(M)\subset E^{n}$ be the space of the helix directions of $%
M$. If $\alpha :I\rightarrow M$ is a line of curvature with respect to the
normal component $N_{j}\in \vartheta (M)$ of a direction $d_{j}\in H(M)$ and
if $N_{j}^{%
{\acute{}}%
}\in \varkappa (M)$ along the curve $\alpha $, then $d_{j}\in Sp\left\{
T\right\} ^{\bot }$ along the curve $\alpha $, where $T$ \ is the unit
tangent vector field of $\alpha $.
\end{theorem}

\begin{proof}
We assume that $\alpha :I\rightarrow M$ is a line of curvature with respect
to the normal component $N_{j}\in \vartheta (M)$ of a direction $d_{j}\in
H(M)$. Since $M$ is a strong $r$-helix submanifold, we can decompose $%
d_{j}\in H(M)$ in its tangent and normal components:%
\begin{equation*}
d_{j}=\cos (\theta _{j})N_{j}+\sin (\theta _{j})T_{j}\text{ }
\end{equation*}%
where $\theta _{j}$ is constant. So, $\left\langle N_{j},d_{j}\right\rangle
= $constant and by taking derivatives on both sides along the curve $\alpha $%
, we get $\left\langle N_{j}^{%
{\acute{}}%
},d_{j}\right\rangle =0$. On the other hand, since $\alpha :I\rightarrow M$
is a line of curvature with respect to the $N_{j}\in \vartheta (M)$,%
\begin{equation*}
A_{N_{j}}(T)\text{=tang}(-\overline{D}_{T}N_{j})=\text{tang}(-N_{j}^{%
{\acute{}}%
})=\lambda _{j}T\text{ }
\end{equation*}%
along the curve $\alpha $. According to this Theorem, $N_{j}^{%
{\acute{}}%
}\in \varkappa (M)$ along the curve $\alpha $. Hence,%
\begin{equation}
\text{tang}(-N_{j}^{%
{\acute{}}%
})=-N_{j}^{%
{\acute{}}%
}=\lambda _{j}T\text{ }
\end{equation}%
Therefore, by using the equalities $\left\langle N_{j}^{%
{\acute{}}%
},d_{j}\right\rangle =0$ and (3.12), we obtain:%
\begin{equation*}
\left\langle T,d_{j}\right\rangle =0
\end{equation*}%
along the curve $\alpha $. This completes the proof.
\end{proof}

\begin{theorem}
Let $M$ be a submanifold with $(n-k)$ dimension in $E^{n}$ and let $%
\overline{D}$ be Riemannian connexion (standart covariant derivative) on $%
E^{n}$. Let us assume that $M\subset E^{n}$ be a strong $r$-helix
submanifold and $H(M)\subset E^{n}$ be the space of the helix directions of $%
M$. If $\alpha :I\rightarrow M$ is a curve in $M$ and if the system $\left\{
T_{j}^{%
{\acute{}}%
},T\right\} $ is linear dependent along the curve $\alpha $, where $T_{j}^{%
{\acute{}}%
}$ is the derivative of the tangent component $T_{j}$ of a direction $%
d_{j}\in H(M)$ and $T$ the tangent to the curve $\alpha $, then $\alpha $ is
a line of curvature in $M$.
\end{theorem}

\begin{proof}
Since $M$ is a strong $r$-helix submanifold, we can decompose $d_{j}\in H(M)$
in its tangent and normal components:%
\begin{equation}
d_{j}=\cos (\theta _{j})N_{j}+\sin (\theta _{j})T_{j}\text{ }
\end{equation}%
where $\theta _{j}$ is constant. If we take derivative in each part of the
equation (3.13) along the curve $\alpha $, we obtain:%
\begin{equation}
0=\cos (\theta _{j})N_{j}^{%
{\acute{}}%
}+\sin (\theta _{j})T_{j}^{%
{\acute{}}%
}\text{ }
\end{equation}%
From (3.14), we can write%
\begin{equation}
N_{j}^{%
{\acute{}}%
}=-\tan (\theta _{j})T_{j}^{%
{\acute{}}%
}\text{ }
\end{equation}%
So, for the tangent component of $-N_{j}^{%
{\acute{}}%
}$, from (3.15) we can write:%
\begin{equation}
A_{N_{j}}(T)\text{=tang}(-\overline{D}_{T}N_{j})=\text{tang}(-N_{j}^{%
{\acute{}}%
})=\text{tang}(\tan (\theta _{j})T_{j}^{%
{\acute{}}%
})\text{ }
\end{equation}%
along the curve $\alpha $. According to the hypothesis, the system $\left\{
T_{j}^{%
{\acute{}}%
},T\right\} $ is linear dependent along the curve $\alpha $. Hence, we get $%
T_{j}^{%
{\acute{}}%
}=\lambda _{j}T$. And, by using the equation (3.16), we have:%
\begin{equation*}
A_{N_{j}}(T)=\text{tang}(\tan (\theta _{j})T_{j}^{%
{\acute{}}%
})=\text{tang}(\tan (\theta _{j})\lambda _{j}T)
\end{equation*}%
and%
\begin{equation}
A_{N_{j}}(T)=\text{tang}(\tan (\theta _{j})\lambda _{j}T)\text{ }
\end{equation}%
Moreover, since $T\in \varkappa (M)$, tang$(\tan (\theta _{j})\lambda
T)=(\tan (\theta _{j})\lambda _{j})T=k_{j}T$. Therefore, from (3.17), we
have:%
\begin{equation*}
A_{N_{j}}(T)=k_{j}T\text{.}
\end{equation*}%
It follows that $\alpha $ is a line of curvature in $M$ for $N_{j}\in
\vartheta (M)$. This completes the proof.
\end{proof}

\noindent \textbf{Acknowledgment.} The authors would like to thank two
anonymous referees for their valuable suggestions and comments that helped
to improve the presentation of this paper.

\end{document}